\documentclass[12pt]{amsart}
\usepackage{amsmath}
\usepackage{amssymb}
\usepackage{hyperref}
\usepackage[all]{xy}
\SelectTips{eu}{}
\usepackage{fullpage}
\usepackage[mathscr]{eucal}
\usepackage{mathtools}
\allowdisplaybreaks
\numberwithin{equation}{section}
\newtheorem{lemma}[equation]{Lemma}
\newtheorem{theorem}[equation]{Theorem}
\newtheorem{prop}[equation]{Proposition}
\newtheorem{cor}[equation]{Corollary}

\theoremstyle{definition}
\newtheorem{defn}[equation]{Definition}

\newtheorem{rk}[equation]{Remark}
\newtheorem{rks}[equation]{Remarks}

\newcommand{\bF}{\mathbb F}
\newcommand{\bC}{\mathbb C}
\newcommand{\bQ}{\mathbb Q}
\newcommand{\bR}{\mathbb R}
\newcommand{\bZ}{\mathbb Z}

\newcommand{\fb}{\mathfrak b}

\newcommand{\fH}{\mathfrak{H}}
\newcommand{\fK}{\mathfrak{K}}

\newcommand{\fq}{\mathfrak q}

\newcommand{\bi}{\mathrm{i}}
\newcommand{\mra}{\mathrm{a}}
\newcommand{\mrb}{\mathrm{b}}
\newcommand{\mrc}{\mathrm{c}}
\newcommand{\mrd}{\mathrm{d}}
\newcommand{\Aut}{\mathsf{Aut}}
\newcommand{\Diag}{\mathsf{Diag}}

\newcommand{\Gasp}{\Gamma}
\newcommand{\mcg}[1]{\Gamma_{#1}}

\newcommand{\Hom}{\mathsf{Hom}}

\newcommand{\Inn}{\mathsf{Inn}}

\newcommand{\Mp}{\mathsf{Mp}}
\newcommand{\Orth}{\mathsf{O}}
\newcommand{\Out}{\mathsf{Out}}

\newcommand{\sig}{\text{\rm signature}}
\newcommand{\Sp}{\mathsf{Sp}}
\newcommand{\Spq}{\Sp^{\mathsf q}}
\newcommand{\fsp}{\mathfrak{sp}}

\renewcommand{\leq}{\leqslant}
\renewcommand{\ge}{\geqslant}

\newcommand{\UU}{\mathsf{U}}

\newcommand{\Tr}{\mathsf{Tr}}

\title{Cohomology of symplectic groups and Meyer's signature theorem}

\author{Dave Benson}
\address{Institute of Mathematics, University of Aberdeen, Aberdeen
  AB24 3UE, Scotland, United Kingdom}
\author{Caterina Campagnolo}
\address{Department of Mathematics, Karlsruhe Institute of Technology,
D-76128 Karlsruhe, Germany}
\author{Andrew Ranicki}
\address{School of Mathematics, University of Edinburgh, Edinburgh EH9
  3FD, Scotland, United Kingdom}
\author{Carmen Rovi}
\address{Department of Mathematics, Indiana University, Bloomington IN 47405, USA}

\begin{document}

\begin{abstract}
Meyer showed that the signature of a closed oriented surface bundle over a surface 
is a multiple of $4$, and can be computed using an element of
$H^2(\Sp(2g, \bZ),\bZ)$. 
If we denote  by $1 \to \bZ \to \widetilde{\Sp(2g,\bZ)} \to \Sp(2g,\bZ) \to 1$ the pullback 
of the universal cover of $\Sp(2g,\bR)$, then by a theorem of 
Deligne, every finite index subgroup 
of $\widetilde{\Sp(2g, \bZ)}$ contains $2\bZ$. As a consequence, a class in the 
second cohomology of any finite quotient of $\Sp(2g, \bZ)$ can at most enable
us to compute the signature of a surface bundle modulo $8$. We show that this 
is in fact possible and investigate the smallest quotient of $\Sp(2g, \bZ)$ that 
contains this information. This quotient $\fH$ is a nonsplit extension of $\Sp(2g,2)$ 
by an elementary abelian group of order $2^{2g+1}$. There is a 
central extension $1\to \bZ/2\to\tilde{\fH}\to\fH\to 1$, and $\tilde{\fH}$ 
appears as a quotient of the metaplectic double cover 
$\Mp(2g,\bZ)=\widetilde{\Sp(2g,\bZ)}/2\bZ$. It is an extension of 
$\Sp(2g,2)$ by an almost extraspecial group of order $2^{2g+2}$, and 
has a faithful irreducible complex representation of dimension $2^g$.
Provided $g\ge 4$, the extension $\widetilde{\fH}$ is the universal central
extension of $\fH$.
Putting all this together, in Section \ref{se:sigmod8} we provide a recipe for
computing the signature modulo $8$, and indicate some consequences.
\end{abstract}

\subjclass[2010]{20J06 (primary); 55R10, 20C33 (secondary)}

\maketitle

\section{Introduction}

Let $\Sigma_g \to M \to \Sigma_h$ be an oriented surface bundle over a surface. 
This is determined by a homotopy class of maps 
$\Sigma_h \to B\Aut^+(\Sigma_g)$. If $g\ge 2$ then
the connected components of $\Aut^+(\Sigma_g)$ are contractible
(Corollary 19 in Luke and Mason~\cite{Luke/Mason:1972a}; see also 
Earle and Eells~\cite{Earle/Eells:1969a} and Hamstrom~\cite{Hamstrom:1966a}), and 
$\pi_0\,\Aut^+(\Sigma_g)=\Gamma_g$ is the (orientation preserving) 
mapping class group of $\Sigma_g$. So $B\Aut^+(\Sigma_g)\simeq B\Gamma_g$,
and the bundle is classified by a homotopy class of maps 
$\Sigma_h\to B\Gamma_g$, or equivalently by the monodromy homomorphism
\[ \pi_1(\Sigma_h) = \langle a_1,b_1,\dots,a_h,b_h\mid
  [a_1,b_1]\dots[a_h,b_h]=1\rangle \to \Gamma_g. \]
Now $\Gamma_g$ acts on $H^1(\Sigma_g,\bZ)\cong \bZ^{2g}$ preserving
the symplectic form given by cup product into $H^2(\Sigma_g,\bZ)\cong
\bZ$. So we have a map $\Gamma_g \to \Sp(2g,\bZ)$, which is surjective.
Composing, we obtain a map
\[ \chi\colon \pi_1(\Sigma_h) \to \Gamma_g \to \Sp(2g,\bZ), \]
and an induced map in cohomology
\[ \chi^* \colon H^2(\Sp(2g,\bZ),\bZ) \to H^2(\pi_1(\Sigma_h),\bZ). \]
Meyer \cite{Meyer:1973a} constructed a $2$-cocycle $\tau$ on
$\Sp(2g,\bZ)$ such that 
\[ \sig(M)=\langle \chi^*[\tau],[\Sigma_h]\rangle \in 4\bZ\subseteq \bZ \]
with $[\tau]=4$ in $H^2(\Sp(2g,\bZ),\bZ)\cong \bZ$ for $g\ge 3$, and
$[\tau]/4=1$ in $H^2(\Sp(2g,\bZ),\bZ)\cong \bZ$ classifying the universal 
central extension of $\Sp(2g,\bZ)$.
Let  
\[ ([\tau]/4)_2 \in H^2(\Sp(2g,\bZ),\bZ/2)\cong\bZ/2 \] 
be the mod $2$ reduction. The mod $2$ residue 
\[ \sig(M)/4 = \langle\chi^*[\tau]/4,[\Sigma_h]\rangle
  =\langle\chi^*([\tau]/4)_2,[\Sigma_h]\rangle \in \bZ/2  \]
was identified by Rovi \cite{Rovi:AGT} with the Arf--Kervaire 
invariant of a Pontryagin squaring operation.

Our main purpose in this paper is to construct a normal subgroup 
$\fK$ of $\Sp(2g,\bZ)$  for $g \geq 1$ with finite quotient 
$\fH=\Sp(2g,\bZ)/\fK$ of shape  $2^{2g+1}\,^{\boldsymbol\cdot}\,\Sp(2g,2)$
(for notation describing group extensions, see Section 5.2 of the Introduction to the Atlas \cite{Atlas}). 
Let $p\colon\Sp(2g,\bZ)\to\fH$ be the projection.
There is a nonzero element $c \in H^2(\fH,\bZ/2)$
(for $g\ge 4$ we have $H^2(\fH,\bZ/2)\cong\bZ/2$ 
but there are extraneous summands inflated from $H^2(\Sp(2g,2),\bZ/2)$ for small $g$) which
classifies a nonsplit double cover $\tilde{\fH}$ of $\fH$. The inflation  
$p^*(c) = [\tau/4]_2$ in $H^2(\Sp(2g,\bZ),\bZ/2)\cong\bZ/2$
classifies the metaplectic double cover $\Mp(2g,\bZ)$ of $\Sp(2g,\bZ)$.
Now $p$  factors through $\Sp(2g,\bZ/4)$ so that 
we obtain as a consequence that $\sig(M)/4 \in \bZ/2$ only 
depends on the $\bZ/4$-coefficient monodromy 
$\chi_4\colon\pi_1(\Sigma_h) \to \Sp(2g,\bZ/4)$
(this was already proved by a different method by 
Korzeniewski \cite{Korzeniewski:2005a}). 
The sequence of group homomorphisms 
\[ \Sp(2g,\bZ) \to \Sp(2g,\bZ/4) \to \fH \to \UU(2^g,\bQ[\bi])/\{\pm 1\} \]
lifts to a sequence of double covers
\[ \Mp(2g,\bZ) \to \widetilde{\Sp(2g,\bZ/4)} \to \tilde\fH \to
  \UU(2^g,\bQ[\bi]) \]
which is used in the recipe of Section \ref{se:sigmod8} for the signature
modulo $8$. The faithful representation $\tilde{\fH}\to \UU(2^g,\bQ[\bi])$ 
is investigated in Benson~\cite{Benson:theta}.

Denote by $\widetilde{\Sp(2g,\bZ)}$ the central extension obtained by
pulling back the universal cover of $\Sp(2g,\bR)$:
\[ \xymatrix{1 \ar[r] & \bZ \ar[r] \ar@{=}[d] &
    \widetilde{\Sp(2g,\bZ)} \ar[r] \ar[d] & \Sp(2g,\bZ) \ar[r] \ar[d]
    & 1 \\ 1 \ar[r] & \bZ \ar[r] & \widetilde{\Sp(2g,\bR)} \ar[r] &
    \Sp(2g,\bR) \ar[r] & 1} \]
Then for $g\ge 4$ the group $\widetilde{\Sp(2g,\bZ)}$ is the universal
central extension of $\Sp(2g,\bZ)$, while for $g=3$ there is an extra
copy of $\bZ/2$ coming from the fact that $\Sp(6,2)$ has an
exceptional double cover (see Lemma \ref{le:H2}).
Note also that the centre of $\Sp(2g,\bZ)$
has order two. The centre of $\widetilde{\Sp(2g,\bZ)}$ is twice
as big as the subgroup $\bZ$ displayed above; it is isomorphic to
$\bZ\times \bZ/2$ if $g$ is even, and $\bZ$ if $g$ is odd.

A theorem of Deligne \cite{Deligne:1978a} implies that
the group $\widetilde{\Sp(2g,\bZ)}$ is not residually finite. Every
subgroup of finite index contains the subgroup $2\bZ$.
To rephrase, every finite quotient of
$\widetilde{\Sp(2g,\bZ)}$ is in fact a finite quotient of the
metaplectic double cover $\Mp(2g,\bZ)$ of $\Sp(2g,\bZ)$ defined by
\[ \xymatrix{&1\ar[d]&1\ar[d]\\
&2\bZ \ar@{=}[r]\ar[d]&2\bZ \ar[d] \\
1 \ar[r] & \bZ \ar[r]\ar[d] & \widetilde{\Sp(2g,\bZ)} \ar[r]
\ar[d] & \Sp(2g,\bZ) \ar[r] \ar@{=}[d] & 1 \\
1 \ar[r] & \bZ/2 \ar[r] \ar[d] & \Mp(2g,\bZ) \ar[r] \ar[d] &
\Sp(2g,\bZ) \ar[r] & 1 \\
& 1 & 1 } \]
As a consequence, if we compose
$\chi$ with the map to a finite quotient of $\Sp(2g,\bZ)$, we
lose information about the signature; the best we can hope to 
do is compute the signature modulo $8$. We shall discuss this
in greater detail elsewhere.

An outline of the paper is as follows. In Section \ref{se:K}, we
describe the subgroup $\fK\leq \Sp(2g,\bZ)$ and quotient
$\fH=\Sp(2g,\bZ)/\fK$. Their properties are
described in Theorem \ref{th:main}, and the proof occupies 
much of the rest of
the paper. Section \ref{se:extraspecial} contains background and
references on extraspecial and almost extraspecial groups, and 
explains what this has to do with the structure of $\fH$ and its
double cover $\tilde\fH$. We describe a faithful unitary representation
$\tilde\fH \to \UU(2^g,\bQ[\bi])$,
which inflates to a representation $\Mp(2g,\bZ) \to \UU(2^g,\bQ[\bi])$,
and which is investigated in greater detail in \cite{Benson:theta}.
Section \ref{se:sigmod8} uses this representation to give a recipe for
computing the signature modulo $8$ of a surface bundle over a
surface. The rest of the paper consists of cohomology computations. In
preparation for this, in Section \ref{se:symp} we discuss the Lie algebra
of the symplectic group. We show that as a module, it is isomorphic to the 
divided square of the natural module, and we discuss the submodule structure.
This enables us in Section \ref{se:coho} to exploit the five-term sequence to
compute $H_2(\fH)$ and $H_2(\Sp(2g,\bZ/2^n))$ for $n\ge 2$. Provided that
$g\ge 4$, these are isomorphic to $\bZ/2$.

\subsection*{Acknowledgements}
Campagnolo acknowledges support by the Swiss National Science Foundation, grant number PP00P2-128309/1, and by the German Science Foundation via the Research Training Group 2229, under which this research was started and then completed.

\section{The subgroup $\fK\leq \Sp(2g,\bZ)$ and the main theorem}\label{se:K}

Denote by $J$ the $2g\times 2g$ matrix
\[ \begin{pmatrix} 0 & I \\ -I & 0 \end{pmatrix}. \]
Regarding $J$ as a symplectic form, we have 
\[ \Sp(2g,\bZ) = \left\{\begin{pmatrix} A&B\\C&D\end{pmatrix} = X 
\mid X^tJX=J\right\}. \]
Since $J^{-1}=-J$, a matrix is symplectic if and only if its transpose
is symplectic. Writing out the above condition explicitly, a matrix is
symplectic if and only if\smallskip

(i) $AB^t$ and $CD^t$ are symmetric, and $AD^t-BC^t=I$, or equivalently\smallskip

(ii) $A^tC$ and $B^tD$ are symmetric, and $A^tD-C^tB=I$.\smallskip

\noindent
We write $\Sp(2g,2)$ for the matrices satisfying the same conditions
over $\bF_2$, and note that reduction modulo two
$\Sp(2g,\bZ)\to\Sp(2g,2)$ is surjective (Newman and Smart \cite{Newman/Smart:1964a}).

We write $\Gamma(2g,N)\leq\Sp(2g,\bZ)$ for the \emph{principal congruence subgroup}
consisting of symplectic matrices which are congruent to the identity
modulo $N$. We write $\Gamma(2g,N,2N)$ for the
\emph{Igusa subgroup} \cite{Igusa:1964a} of $\Gamma(2g,N)$ consisting of the matrices
$\left(\begin{smallmatrix}A&B\\C&D\end{smallmatrix}\right)$
where the entries of $\Diag(AB^t)$ and $\Diag(CD^t)$ are divisible by
$2N$, or equivalently where the entries of $\Diag(A^tC)$ and
$\Diag(B^tD)$ are divisible by $2N$. If $N=1$, this is the \emph{theta subgroup}, also known as the
\emph{symplectic quadratic group}, and denoted $\Spq(2g,\bZ)$. It is the inverse
image in $\Sp(2g,\bZ)$ of the orthogonal subgroup 
$\Orth^+(2g,2)\leq\Sp(2g,2)$.

\begin{defn}
We write $\fK$ for the subgroup of $\Sp(2g,\bZ)$ consisting of
matrices
\[ \begin{pmatrix} I+2\mra & 2\mrb \\ 2\mrc & I+2\mrd \end{pmatrix} \in
  \Sp(2g,\bZ) \]
satisfying the following:
\begin{enumerate}
\item[\rm (i)] The vectors of diagonal entries $\Diag(\mrb)$ and $\Diag(\mrc)$ are even, and
\item[\rm (ii)] the trace $\Tr(a)$ is even.
\end{enumerate}
\end{defn}

Thus we have $\Gamma(2g,4) \leq \fK \leq \Gamma(2g,2)$ and 
$|\Gamma(2g,2):\fK|=2^{2g+1}$. The interpretation of the subgroup 
$\fK$ is that it is the inverse image
in $\Sp(2g,\bZ)$ of the largest subspace of
$\Gamma(2g,2)/\Gamma(2g,4)$ on which 
the quadratic form in Theorem \ref{th:main}(iv) is identically zero.

Our main theorem is as follows. We assume that $g\ge 4$ for  the
purpose of simplifying the statements. In an appendix we include
statements for all values of $g$. The main difference for low values
of $g$ is that the cohomology of $\Sp(2g,2)$ in degrees one and 
two contributes some further annoying complications.

\begin{theorem}\label{th:main}
Let $g\ge 4$. 
\begin{enumerate}
\item[\rm (i)] $\fK$ is a normal subgroup of $\Sp(2g,\bZ)$. We write
$\fH$ for the quotient $\Sp(2g,\bZ)/\fK$.
\item[\rm (ii)] The quotient $\Gamma(2g,2)/\fK \leq \fH$ is an elementary abelian
  $2$-group $(\bZ/2)^{2g+1}$.
\item[\rm (iii)] The extension
\[ 1 \to (\bZ/2)^{2g+1} \to \fH \to \Sp(2g,2) \to 1 \]
does not split.
\item[\rm (iv)] The group $(\bZ/2)^{2g+1}$ supports an invariant
  quadratic form $\fq$ given by
\[ \fq\begin{pmatrix} I+2\mra & 2\mrb \\ 2\mrc & I+2\mrd\end{pmatrix}
=\Tr(\mra)+\langle \Diag(\mrb),\Diag(\mrc)\rangle  \]
(see Remark \ref{rk:qbsp} for the definition of the pointy brackets here).
\item[\rm (v)] The action of $\Sp(2g,2)$ on $(\bZ/2)^{2g+1}$
  described by the extension in  {\rm (iii)}
  gives the exceptional isomorphism $\Sp(2g,2) \cong \Orth(2g+1,2)$,
 the orthogonal group of the quadratic form $\fq$.
\item[\rm (vi)] We have
$H^2(\fH,\bZ/2) \cong \bZ/2$, and an associated central extension
\[ 1 \to \bZ/2 \to \tilde\fH \to \fH \to 1. \]
\item[\rm (vii)] For $n\ge 2$, the inflation map $H^2(\fH,\bZ/2)\to
  H^2(\Sp(2g,\bZ/2^n),\bZ/2)$ is an isomorphism.
\item[\rm (viii)] The nonzero element of $H^2(\Sp(2g,\bZ)/\fK,\bZ/2)$
  inflates to the reduction modulo two of $\frac{1}{4}[\tau]$ as an
  element of $H^2(\Sp(2g,\bZ),\bZ/2)$.
\item[\rm (ix)] Restricting the central extension of $\fH$ to the
  subgroup $\Gamma(2g,2)/\fK$ gives an almost extraspecial group
  $2^{1+(2g+1)}\leq \tilde\fH$.
\end{enumerate}
\end{theorem}

The proof of this theorem occupies the rest of the paper.

\section{Extraspecial and almost extraspecial groups}\label{se:extraspecial}

For background on extraspecial and almost extraspecial groups, we
refer the reader to Section I.5.5 of Gorenstein \cite{Gorenstein:1968a} and 
Section III.13 of Huppert \cite{Huppert:1967a}, as well as the papers of
Bouc and Mazza~\cite{Bouc/Mazza:2004a},
Carlson and Th\'evenaz~\cite{Carlson/Thevenaz:2000a},
Glasby~\cite{Glasby:1995a}, Griess~\cite{Griess:1973a}, Hall and
Higman~\cite{Hall/Higman:1956a}, 
Lam and Smith~\cite{Lam/Smith:1989a},
Quillen~\cite{Quillen:1971d},
Schmid~\cite{Schmid:2000a},
Stancu~\cite{Stancu:2002a}, and the letter from Isaacs to Diaconis 
reproduced in the appendix of Diaconis~\cite{Diaconis:2010a}.

The cohomology ring $H^*((\bZ/2)^n,\bZ/2)$ is a polynomial ring in
generators $z_1,\dots,z_n$ of degree one. Thus
\[ H^1((\bZ/2)^n,\bZ/2)=\Hom((\bZ/2)^n,\bZ/2) \]
is an $n$-dimensional vector space spanned by the linear forms
$z_1,\dots,z_n$. An element of degree two is therefore a quadratic
form $\fq$ on $(\bZ/2)^n$. Letting $\fb$ be the associated bilinear
form $(\bZ/2)^n \times (\bZ/2)^n \to \bZ/2$, we have
\[ \fq(x+y) = \fq(x) + \fq(y) + \fb(x,y). \]

In the corresponding central extension
\[ 1 \to \bZ/2 \to E \to (\bZ/2)^n \to 1 \]
the role played by $\fq$ and $\fb$ 
is as follows. If $x$ and $y$ are elements of $(\bZ/2)^n$,
choose preimages $\hat x$ and $\hat y$ in $E$. Then 
as elements of the central $\bZ/2$, we have $\hat x^2=\fq(x)$ and
$[\hat x,\hat y] = \fb(x,y)$.

\begin{defn} 
We say that a quadratic form $\fq$ is \emph{nonsingular} if the
radical $\fb^\perp$ of the associated bilinear form $\fb$ is $\{0\}$, and
\emph{nondegenerate} if $\fb^\perp \cap \fq^{-1}(0) = \{0\}$.
\end{defn}

If $\fq$ is nonsingular then $n=2g$ is even; in this case there are two
isomorphism classes of quadratic forms, distinguished by the Arf
invariant. The corresponding groups $E$ defined by the central extension
\[ 1 \to \bZ/2 \to E \to (\bZ/2)^{2g} \to 1  \]
are called \emph{extraspecial $2$-groups}, and are characterised by the
properties 
\[ \Phi(E)=[E,E] = Z(E) \cong \bZ/2. \]
The two isomorphism classes of extraspecial groups are denoted $2^{1+2g}_+$ (Arf invariant
zero) and $2^{1+2g}_-$ (Arf invariant one).

If $\fq$ is singular but nondegenerate then $n=2g+1$ is odd; in this
case there is one isomorphism class of quadratic forms. The
corresponding groups $E$ defined by the central extension
\[ 1 \to \bZ/2 \to E \to (\bZ/2)^{2g+1} \to 1 \]
are called \emph{almost extraspecial groups}. 
The central product of $\bZ/4$ with an extraspecial group of either
isomorphism type of order $2^{1+2g}$ gives the almost extraspecial group
of order $2^{1+(2g+1)}$.

If $G$ is a group, we write $\Aut(G)$ for the group of automorphisms
of $G$, we write $\Out(G)$ for the group of outer automorphisms, and we write $\Inn(G)$
for the group of inner automorphisms. These fit into short exact
sequences
\begin{gather*}
1 \to Z(G) \to G \to \Inn(G) \to 1 \,\mbox{ and }\,
1 \to \Inn(G) \to \Aut(G) \to \Out(G) \to 1. 
\end{gather*}
Writing the automorphism groups of the extraspecial and almost extraspecial  
groups as extensions of the outer by the inner automorphisms in this way, we have sequences
\begin{gather} 
1 \to (\bZ/2)^{2g} \to \Aut(2^{1+2g}_+) \to \Orth^+(2g,2) \to 1, \notag \\
1 \to (\bZ/2)^{2g} \to \Aut(2^{1+2g}_-) \to \Orth^-(2g,2) \to 1, \notag \\
1 \to (\bZ/2)^{2g} \to \Aut(2^{1+(2g+1)}) \to \Sp(2g,2) \times
\bZ/2 \to 1, \label{eq:Aut-almost-exspec}
\end{gather}
which do not split provided $g\ge 4$. It is the last case that is of
interest to us: in this case the extra
factor of $\bZ/2$ in the outer automorphism group $\Out(E)$ acts by inverting the central
element of order four, and for $g\ge 3$ the derived subgroup
$\Out(E)'$ is $\Sp(2g,2)$.

It was proved by Griess \cite{Griess:1973a} using 
 representation theory, that in each case, there is an extension of
the extraspecial group by its outer automorphism group, and of the
almost extraspecial group by the subgroup of index two in its outer
automorphism group.

We are interested in the almost extraspecial case. In this case, what
Griess proved (part (b) of Theorem 5 of \cite{Griess:1973a}) is that
there is a group which he denotes $H_0$ of shape $2^{1+(2g+1)}\Sp(2g,2)$, with the
following properties. 
The normal $2$-subgroup $O_2(H_0)$ is the almost extraspecial group
$2^{1+(2g+1)}$, and the quotient $H_0/Z(H_0)$
is isomorphic to the subgroup of index two in $\Aut(2^{1+(2g+1)})$.

Dempwolff \cite{Dempwolff:1974a} proved that for $g\ge 2$ there is a
unique isomorphism class of nonsplit extensions of $\Sp(2g,2)$ by 
an elementary abelian group $(\bZ/2)^{2g}$ with nontrivial action. We
shall combine the results of Griess and Dempwolff to show that the
group $\fH$ of Theorem \ref{th:main} is isomorphic to the quotient of 
Griess' group $H_0$ by the central subgroup of order two. This in turn
allows us to compute $H^2(\fH,\bZ/2)$ and relate it to $H^2(\Sp(2g,\bZ),\bZ)$.

There is another approach to this, which we describe in a separate
paper \cite{Benson:theta}.
This avoids the use of the
theorems of Griess and Dempwolff, replacing them with a computation 
showing that the group $\tilde \fH$ has a 
Curtis--Tits--Steinberg type presentation. This approach is closely
related to the action of $\tilde \fH$ on a certain $2^g$ dimensional space
of theta functions, and shows that the following defines a 
projective representation $\sigma\colon\Sp(2g,\bZ)\to
\UU(2^g,\bQ[\bi])/\{\pm I\}$ with kernel $\fK$, and then induces
a $2^g$ dimensional representation $\tilde \fH\to\UU(2^g,\bQ[\bi])$.

The underlying vector space for the representation 
has as a basis the vectors $e_w$ for $w\in\{0,1\}^g$.
In the following matrices, we regard $\det A$,
which is really an element of $(\bZ/4)^\times=\{1,-1\}$, as
being either $+1$ or $-1$ in $\bC$, and $\sqrt{\det A}$ is either 
$1$ or $\bi$.
\begin{align*}
\sigma
\begin{pmatrix} I & B \\ 0 & I \end{pmatrix} &\colon
e_w\mapsto \bi^{w^tBw}e_w,   \\
\sigma\begin{pmatrix} A & 0 \\ 0 & (A^t)^{-1} \end{pmatrix} &\colon
e_w\mapsto \sqrt{\det A}\ e_{(A^t)^{-1}w},  \\
\sigma\begin{pmatrix} 0 & I \\ -I & 0 \end{pmatrix} &\colon
e_w\mapsto \frac{1}{(1-\bi)^g}
\sum_{w'}(-1)^{w^tw'}e_{w'}. 
\end{align*}
Note that $\Sp(2g,\bZ)$ is generated by these elements, but it is not
at all obvious that the relations in $\Sp(2g,\bZ)$ hold up to sign for the linear
transformations listed here; this is proved in \cite{Benson:theta}. 
Note also that in the first formula
above, the matrix $B$ may be interpreted as having diagonal entries in
$\bZ/4$ and off-diagonal entries in $\bZ/2$, so that it represents a
quadratic form on $(\bZ/2)^g$, taking values in $\bZ/4$.

Further references for the representation described here
include Funar and Pitsch~\cite{Funar/Pitsch:preprint}, 
Glasby~\cite{Glasby:1995a}, 
Gocho~\cite{Gocho:1990a,Gocho:1990b},
Nebe, Rains and Sloane~\cite{Nebe/Rains/Sloane:2001a}, 
Runge~\cite{Runge:1993a,Runge:1995a,Runge:1996a}, 
and Tsushima~\cite{Tsushima:2003a}.

\section{Signature modulo eight}\label{se:sigmod8}

Given an oriented surface bundle over a surface $\Sigma_g \to M \to
\Sigma_h$, recall that there is an associated map $\chi\colon
\pi_1(\Sigma_h) \to \Sp(2g,\bZ)$. Composing with $\sigma \colon
\Sp(2g,\bZ) \to \UU(2^g,\bQ[\bi])/\{\pm I\}$, we obtain a map
\[ \phi \colon \pi_1(\Sigma_h)  = \langle a_1,b_1,\dots,a_h,b_h\mid
  [a_1,b_1]\dots[a_h,b_h]=1\rangle \to \UU(2^g,\bQ[\bi])/\{\pm I\}. \]
Now the commutators $[\phi(a_i),\phi(b_i)]$ are well defined in 
$\UU(2^g,\bQ[\bi])$, since changing the sign on $\phi(a_i)$ or
$\phi(b_i)$ changes the sign twice in the commutator. Since
the product of the commutators is in the kernel of $\phi$, we have
\[ [\phi(a_1),\phi(b_1)]\dots[\phi(a_h),\phi(b_h)] = \pm I \in
  \UU(2^g,\bQ[\bi]). \]

\begin{theorem}
We have
\[ [\phi(a_1),\phi(b_1)]\dots[\phi(a_h),\phi(b_h)] = 
\begin{cases} 
\phantom{-}I & \text{\rm iff }\sig(M) \equiv 0 \pmod{8}, \\
-I & \text{\rm iff }\sig(M) \equiv 4 \pmod{8}. \end{cases} \]
\end{theorem}

\begin{rks}
\begin{enumerate}
\item[\rm (1)] As a method of computation, this theorem is not very
  useful, because of the large size of the matrices involved. 
  Endo \cite{Endo:1998a} provided a much more efficient and purely
  algebraic method for computing the signature, and not just modulo
  eight. On the other hand, there are consequences of the theorem that
  are not very apparent from the point of view of Endo's method.
\item[\rm (2)] 
The following is a consequence of the theta function point of view,
and will be discussed in a separate paper \cite{Benson:theta}.
Let $\Spq(2g,\bZ)$ be the theta subgroup of
  $\Sp(2g,\bZ)$.  If the image of
$\chi$ lies in $\Spq(2g,\bZ)$ then we have $\sig(M)\equiv 0\pmod{8}$.
In particular, this holds if the action of $\pi_1(\Sigma_h)$ on
$H^1(\Sigma_g,\bZ/2)$ is trivial. This proves a special case of the
Klaus--Teichner conjecture; see the introduction to 
\cite{Hambleton/Korzeniewski/Ranicki:2007a} for details.
\item[\rm (3)] Consider next the subgroup consisting
  of the matrices
$\left(\begin{smallmatrix}
    A&B\\C&D\end{smallmatrix}\right)\in\Sp(2g,\bZ)$
such that the entries of $C$ are even, and those of $\Diag(C)$ are divisible
by four. If the image of $\chi$ lies in this subgroup then again we
have $\sig(M)\equiv 0 \pmod{8}$. This will be proved in \cite{Benson:theta}.
\end{enumerate}
\end{rks}

\section{Symplectic groups and their Lie algebras}\label{se:symp}

Let $R$ be a commutative ring, and
$\Sp(2g,R)$ be the symplectic group of dimension 
$2g$ over $R$. Explicitly, this consists of matrices $X$
with entries in $R$, and satisfying $X^tJX=J$, where $J$ is the symplectic form
\[ J=\begin{pmatrix} 0 & I \\ -I & 0 \end{pmatrix} \]
and $I$ is a $g\times g$ identity matrix. Denoting by $V_R$ a free $R$-module
of rank $g$, and setting
\[ W_R=V_R^*=\Hom_R(V_R,R), \] 
the matrices $X$ act on $U_R=V_R \oplus W_R$, preserving the skew-symmetric 
bilinear form
\[ \langle\quad,\quad\rangle\colon U_R \times U_R \to R \]
given by
\[ \langle(v,w),(v',w')\rangle = w'(v)-w(v'). \]
For the action of matrices in $\Sp(2g,R)$, 
we regard $(v,w)$ as a column vector of length $2g$ with entries in $R$.
The skew-symmetric bilinear form induces an isomorphism from $U_R$ to $U_R^*$ sending
$u$ to $\langle u,\quad\rangle$. If $R=\bF_2$, we shall write $U$, $V$
and $W$ instead of $U_{\bF_2}$, $V_{\bF_2}$ and $W_{\bF_2}$.

The Lie algebra $\fsp(2g,R)$ consists of matrices $Y$ with entries in $R$,
and satisfying 
\[ JY+Y^tJ=0. \] 
Thus
\[ Y = \begin{pmatrix} \mra & \mrb \\ \mrc & -\mra^t \end{pmatrix}, \]
where $\mrb$ and $\mrc$ are symmetric. 
To say that $\mrb$ is symmetric is to say that as an element of 
\[ \Hom_R(W_R,V_R) \cong V_R \otimes_R V_R \]
it is invariant under the transposition swapping the two tensor factors. Thus
$\mrb$ is an element of the divided square $D^2(V_R)$ (which may not be identified
with the symmetric square $S^2(U_R)$ unless $2$ happens to be invertible in $R$,
which will not be the case for us). Similarly, we have $\mrc\in D^2(W_R)$ and
\[ \mra \in \Hom_R(V_R,V_R) \cong V_R \otimes_R W_R . \]
Putting this together, we see that 
\[ Y \in D^2(V_R) \oplus D^2(W_R) \oplus (V_R \otimes_R W_R) \cong D^2(U_R). \]
Thus, as a module for $\Sp(2g,R)$, we have identified the Lie algebra $\fsp(2g,R)$ with the divided square of
the natural module. More abstractly, 
if $u\in U_R$ then the symmetric tensor
$u\otimes u$ is identified 
with the endomorphism sending $x$ to $\langle u,x\rangle u$. Polarising, this identifies
$u\otimes u' + u' \otimes u$ with the endomorphism of $U_R$ sending
$x$ to $\langle u,x\rangle u' + \langle u',x\rangle u$.
We have therefore proved the following.

\begin{theorem}
For any commutative ring $R$, 
we have isomorphisms
\[ \fsp(2g,R)\cong D^2(U_R) \cong R^{g(2g+1)}. \]
The first isomorphism is an isomorphism of $\Sp(2g,R)$-modules, while the
second is an isomorphism of $R$-modules.
\end{theorem}

We are interested in the group $\Sp(2g,\bZ/4)$. This sits in a short exact sequence
\[ 1 \to \fsp(2g,2) \to \Sp(2g,\bZ/4) \to \Sp(2g,2) \to 1. \]
The elementary abelian  $2$-subgroup is identified with
$\fsp(2g,2)$, the symplectic Lie algebra over $\bF_2$, 
and consists of the matrices $I+2Y$ with
$Y\in\fsp(2g,2)$. These have the form
\[ \begin{pmatrix} I+2\mra & 2\mrb \\ 2\mrc & I-2\mra^t \end{pmatrix} \]
with $\mrb$ and $\mrc$ symmetric. We have a short exact sequence
\[ \xymatrix@R=6mm{0 \ar[r] & \Lambda^2(U) \ar[r] & D^2(U) \ar[r] \ar[d]^\cong & U \ar[r] & 0 \\
&&\fsp(2g,2)} \]
where $\Lambda^2(U)$ is spanned by elements of the form $u\otimes u' + u' \otimes u$.
As a submodule of $\fsp(2g,2)$, this consists of the matrices where $\Diag(\mrb)=\Diag(\mrc)=0$. 
The quotient $U$ corresponds to the diagonal entries in $\mrb$ and $\mrc$. Thus the above
short exact sequence can be thought of as a short exact sequence of groups
\[ 1\to \Gasp(2g,2,4)/\Gasp(2g,4) \to \Gasp(2g,2)/\Gasp(2g,4) \to \Gasp(2g,2)/\Gasp(2g,2,4) \to 1. \]
More generally, we have short exact sequences
\[ 1 \to \fsp(2g,2) \to \Sp(2g,\bZ/2^{n+1}) \to \Sp(2g,\bZ/2^n) \to 1 \]
and
\begin{multline*} 
1\to \Gasp(2g,2^n,2^{n+1})/\Gasp(2g,2^{n+1}) \to
\Gasp(2g,2^n)/\Gasp(2g,2^{n+1}) 
\to \Gasp(2g,2^n)/\Gasp(2g,2^n,2^{n+1}) \to 1. 
\end{multline*}

\begin{prop}
As modules over $\Sp(2g,2)$, for $g\ge 1$ and $n\ge 1$, we have
\begin{equation*}  
\Gasp(2g,2^n)/\Gasp(2g,2^n,2^{n+1})\cong U \mbox{ and }\,
\Gasp(2g,2^n,2^{n+1})/\Gasp(2g,2^{n+1}) \cong \Lambda^2(U). 
\end{equation*}
\end{prop}

Now the symplectic form on $U$ gives us a map 
$\Lambda^2(U) \to \bF_2$, sending $u\otimes u' + u' \otimes u$ to $\langle u, u'\rangle$.
We write $Y$ for the kernel of this map, and we write
$Z$ for $D^2(U)/Y$, an $\bF_2$-vector space of dimension $2g+1$. 
Putting these together, we have the following diagram of
modules.
\[ \xymatrix{&0\ar[d]&0\ar[d] \\ &Y\ar[d] \ar@{=}[r] & Y \ar[d] \\
0 \ar[r] & \Lambda^2(U) \ar[r] \ar[d] & D^2(U)\ar[d]
\ar[r]&U\ar@{=}[d]\ar[r] &0\\
0 \ar[r] & \bF_2 \ar[d] \ar[r] & Z \ar[d] \ar[r] & U\ar[r]&0\\
&0&0} \]
We claim that the symplectic form on $U$ lifts to a nondegenerate orthogonal form on $Z$, 
invariant under $\Sp(2g,2)$. The quadratic form $Z \to \bF_2$ is given by
\[ \fq(u \otimes u)=0,\quad \fq(u\otimes u' + u' \otimes u) = \langle u,u'\rangle, \]
and the associated symmetric bilinear form is
\begin{gather*} 
\fb(u\otimes u,u'\otimes u')=\langle u,u'\rangle,\\ 
\fb(u\otimes u' + u'\otimes u,u''\otimes u'')=0,\\
\fb(u\otimes u' + u' \otimes u, u'' \otimes u''' + u''' \otimes u'')=0. 
\end{gather*}
A priori, these are a quadratic form and associated bilinear 
form on $D^2(U)$. But they
clearly vanish identically on $Y$, and define a 
nondegenerate but singular quadratic form and associated 
bilinear form on $Z$. These are invariant under $\Sp(2g,2)$, which is therefore the
orthogonal group on $Z\cong {\bF_2}^{2g+1}$, displaying the isomorphism
\[ \Sp(2g,2)\cong\Orth(2g+1,2). \]

\begin{rk}\label{rk:qbsp}
Translating back from $D^2(U)$ to $\fsp(2g,2)$, the
quadratic and bilinear form are given as follows:
\begin{align*} 
\fq\begin{pmatrix} \mra&\mrb \\ \mrc&-\mra^t\end{pmatrix}
&=\Tr(\mra) +\langle\Diag(\mrb),\Diag(\mrc)\rangle,
\\
\fb\left(\begin{pmatrix} \mra&\mrb \\ \mrc&-\mra^t\end{pmatrix},
\begin{pmatrix} \mra'&\mrb' \\ \mrc'&-\mra'^t\end{pmatrix}\right)
&=\langle\Diag(\mrb),\Diag(\mrc')\rangle+\langle\Diag(\mrb'),\Diag(\mrc)\rangle.
\end{align*}
Here, the pointy brackets denote the standard inner product on
${\bF_2}^g$ given by multiplying corresponding coordinates and summing.

The normal subgroup $\fK$ described in Section \ref{se:K} 
is the inverse image of 
\[ Y\leq \fsp(2g,2)\leq \Sp(2g,\bZ/4) \] 
under the quotient map $\Sp(2g,\bZ)\to \Sp(2g,\bZ/4)$. Thus there is a
short exact sequence
\[ 1 \to Z \to \fH \to \Sp(2g,2) \to 1 \]
and the subgroup $Z\cong (\bZ/2)^{2g+1}$ may be viewed as 
the orthogonal module ${\bF_2}^{2g+1}$ for $\Sp(2g,2)$ via conjugation.
\end{rk}

\begin{rk}
The submodule structure of the $\bF_2\Sp(2g,2)$-modules 
$\Lambda^2(U)$ of dimension $g(2g-1)$ and 
$D^2(U)\cong \fsp(2g,2)$ of dimension $g(2g+1)$ can be described explicitly as
follows (see also Hiss \cite{Hiss:1984a}).
There is a map $\Lambda^2(U) \to \bF_2$
corresponding to the symplectic form, given by
\[ u\otimes u' + u' \otimes u \mapsto \langle u,u'\rangle. \]
There is a dual map
$\bF_2\to \Lambda^2(U)$ coming from the fact that the representation
$\Lambda^2(U)$ is self-dual. In terms of the natural bases $v_1,\dots,v_g$
of $V$ and $w_1,\dots,w_g$ of $W$, this is given by
\[ 1 \mapsto \sum_i( v_i\otimes w_i + w_i \otimes v_i). \]

If $g=1$ then $\Lambda^2(U)\cong \bF_2$ is one dimensional, $Y=0$, and 
$Z=D^2(U)$ decomposes as a direct sum $\bF_2 \oplus U$.

If $g\ge 2$ is even then the composite $\bF_2\to \Lambda^2(U)\to \bF_2$ is zero,
and the quotient of the kernel by the image
is a simple module $S$ of dimension $g(2g-1)-2$. Thus $\Lambda^2(U)$ is uniserial
(i.e., it has a unique composition series) with composition factors $\bF_2$, $S$, $\bF_2$.

If $g\ge 3$ is odd, then the composite is nonzero, and 
$\Lambda^2(U)$ decomposes as a direct sum of a trivial module $\bF_2$
and a simple module $S$ of dimension $g(2g-1)-1$. 

In both cases with $g\ge 2$, the unique maximal submodule of $D^2(U)$ is $\Lambda^2(U)$.
We can therefore draw diagrams for the structure of
$D^2(U)\cong \fsp(2g,2)$ as follows:
\[ g=1\colon\quad \bF_2\oplus U\qquad\qquad
g\ge 2\textrm{\ even\ }\colon \quad 
\vcenter{\xymatrix@=3mm{U\ar@{-}[d] \\ \bF_2 \ar@{-}[d] \\ S\ar@{-}[d] \\ \bF_2}} 
\qquad\qquad
g\ge 3\textrm{\ odd\ }\colon \quad 
\vcenter{\xymatrix@=3mm{&U \ar@{-}[dl] \ar@{-}[dr] \\ \bF_2 & & S}}
\]
For $g\ge 2$, the quotient $Z$ of $D^2(U)$ has structure
\[ \vcenter{\xymatrix@=3mm{U\ar@{-}[d] \\ \bF_2}} \]
and this is the orthogonal module for $\Sp(2g,2)\cong \Orth(2g+1,2)$.
The submodule $Y$ is $S$ for $g\ge 3$ odd, it is a nonsplit extension
\[ 0 \to \bF_2 \to Y \to S \to 0 \]
for $g\ge 2$ even, and $Y=0$ for $g=1$.
\end{rk}

\begin{lemma}\label{le:H0Y}
\begin{enumerate}
\item[\rm (i)]
$H_0(\Sp(2g,2),Y)=0$ and $H_0(\Sp(2g,2),U)=0$ for $g\ge 1$.
\item[\rm (ii)]
$H_0(\Sp(2g,2),Z)=0$ and
$H_0(\Sp(2g,2),\fsp(2g,2))=0$ for $g\ge 2$.
\end{enumerate}
\end{lemma}
\begin{proof}
This follows immediately from the structure of $Y$, $U$, $Z$, and 
$\fsp(2g,2)$ as $\Sp(2g,2)$-modules
given in the above remark, since these modules admit no nontrivial homomorphisms
to $\bF_2$ with trivial action.
\end{proof}

\section{Computations in degree two homology and cohomology}\label{se:coho}

\begin{lemma}\label{le:H2Sp2gF2}
\begin{enumerate}
\item[\rm (i)]
$H_2(\Sp(2g,2))=0$ for $g\ge 4$ and
$H_2(\Sp(6,2))\cong\bZ/2$.
\item[\rm (ii)]
$H^2(\Sp(2g,2),\bZ/2)=0$ for $g\ge 4$ and 
$H^2(\Sp(6,2),\bZ/2)\cong\bZ/2$.
\end{enumerate}
\end{lemma}
\begin{proof}
(i) This is computed in the paper of Steinberg \cite{Steinberg:1981a}.

(ii) This follows from the universal coefficient theorem, because
$\Sp(2g,2)$ is perfect for $g\ge 3$.
\end{proof}

\begin{lemma}\label{le:H0SpH1Gamma}
$H_0(\Sp(2g,\bZ/{2^n}),
H_1(\Gasp(2g,2^n)))=0$ for $n\ge 1$ and $g \ge 2$.
\end{lemma}
\begin{proof}
Proposition 10.1 of Sato \cite{Sato:2010a} computes $H_1(\Gasp(2g,N))$, the abelianisation of $\Gasp(2g,N)$, finding that the derived subgroup is $\Gasp(2g,N^2)$ if $N$ is odd,
and $\Gasp(2g,N^2,2N^2)$ if $N$ is even.

Taking $N=2^n$, this gives
\[ H_1(\Gasp(2g,2^n))\cong \Gasp(2g,2^n)/\Gasp(2g,2^{2n},2^{2n+1}). \]
As modules over $\Sp(2g,\bZ/{2^n})$ we have
\begin{gather*}
\Gasp(2g,2^n)/\Gasp(2g,2^{2n})\cong \fsp(2g,\bZ/{2^n}) \,\mbox{ and }\,
\Gasp(2g,2^{2n})/\Gasp(2g,2^{2n},2^{2n+1})\cong U
\end{gather*}
(see Section~\ref{se:symp}). This gives us a short exact sequence
\begin{equation}\label{eq:H0H1}
0 \to U \to H_1(\Gasp(2g,2^n)) \to \fsp(2g,\bZ/{2^n}) \to 0.
\end{equation}
We also have short exact sequences
\begin{equation*}
0 \to \fsp(2g,2) \to \fsp(2g,\bZ/{2^n}) \to \fsp(2g,\bZ/{2^{n-1}}) \to 0.
\end{equation*}
By Lemma \ref{le:H0Y}\,(ii), for $g\ge 2$, we have
\[ H_0(\Sp(2g,\bZ/{2^n}),\fsp(2g,2))\cong H_0(\Sp(2g,2),\fsp(2g,2))=0, \]
and so by induction on $n$ and right exactness of $H_0$, we have
\[ H_0(\Sp(2g,\bZ/{2^n}),\fsp(2g,\bZ/{2^n}))=0. \]
Finally, by Lemma \ref{le:H0Y}\,(i) we have
\begin{equation*}
H_0(\Sp(2g,\bZ/{2^n}),U)\cong H_0(\Sp(2g,2),U)=0.
\end{equation*}
Therefore, using right exactness of $H_0$ on the sequence
\eqref{eq:H0H1}, the lemma is proved.
\end{proof}

\begin{prop}\label{pr:surj1}
For $n\ge 1$ and $g\ge 2$,
\begin{enumerate} 
\item[\rm (i)]
the map $H_2(\Sp(2g,\bZ)) \to H_2(\Sp(2g,\bZ/2^n))$ is surjective, and
\item[\rm (ii)]
the map $H_2(\Sp(2g,\bZ/2^{n+1})) \to H_2(\Sp(2g,\bZ/2^n))$ is surjective.
\end{enumerate}
\end{prop}
\begin{proof}
(i)
The short exact sequence
\[ 1 \to \Gasp(2g,2^n) \to \Sp(2g,\bZ) \to \Sp(2g,\bZ/2^n) \to 1 \]
gives rise to a five-term sequence in homology
\begin{multline*}
H_2(\Sp(2g,\bZ)) \to H_2(\Sp(2g,\bZ/{2^n})) \to H_0\left(\Sp(2g,\bZ/2^n),H_1(\Gasp(2g,2^n))\right) \\
\to H_1(\Sp(2g,\bZ)) \to H_1(\Sp(2g,\bZ/{2^n})) \to 0.
\end{multline*}
The proposition therefore follows immediately from Lemma \ref{le:H0SpH1Gamma}.

(ii) This is similar, observing that $H_1(\Gamma(2g,2^n)/\Gamma(2g,2^{n+1}))
\cong \fsp(2g,2)$, so that by Lemma \ref{le:H0Y}\,(ii) we have 
\begin{equation*}
H_0(\Sp(2g,\bZ/2^n),H_1(\Gamma(2g,2^n)/\Gamma(2g,2^{n+1})))=0.
\qedhere
\end{equation*}
\end{proof}

\begin{cor}
For $n\ge 1$ and $g\ge 3$, the map 
\[ H^2(\Sp(2g,\bZ/{2^n}), A) \to H^2(\Sp(2g,\bZ), A) \] 
is injective for any abelian group of coefficients $A$ with trivial action.
\end{cor}
\begin{proof}
This follows directly from Proposition \ref{pr:surj1} together with 
the universal coefficient theorem for 
cohomology, as the groups $\Sp(2g,\bZ)$ and $\Sp(2g,\bZ/{2^n})$ 
are perfect for $g\ge 3$.
\end{proof}

\begin{prop}\label{pr:surj2}
For $g\ge 2$, the maps $H_2(\Sp(2g,\bZ/4)) \to H_2(\fH)\to
H_2(\Sp(2g,2))$ are surjective.
\end{prop}
\begin{proof}
For the first map, we use the five-term sequence for the short exact sequence
\[ 1 \to Y \to \Sp(2g,\bZ/4) \to \fH \to 1, \] 
and the computation
\[ H_0(\fH,H_1(Y))= H_0(\Sp(2g,2),Y)=0 \]
given in Lemma \ref{le:H0Y}. Note that $Y$ is an elementary abelian
$2$-group, so $H_1(Y)\cong Y$.

The computation for the second map is similar, using the short exact
sequence
\[ 1 \to Z \to \fH \to \Sp(2g,2) \to 1 \]
and the computation $H_0(\Sp(2g,2),Z)=0$ given in Lemma \ref{le:H0Y}.
\end{proof}

\begin{cor}
For $g\ge 3$, the inflation map 
$H^2(\fH, A) \to H^2(\Sp(2g,\bZ/4), A)$ is 
injective for any abelian group of coefficients $A$ with trivial action.
\end{cor}
\begin{proof}
This follows directly from Proposition \ref{pr:surj2} and the universal coefficient theorem for 
cohomology, as the groups $\Sp(2g,\bZ/4)$ are perfect for 
$g\ge 3$, hence all their quotients are perfect as well.
\end{proof}

\begin{prop}\label{pr:barH0}
For $g\ge 4$ the group $\fH=\Sp(2g,\bZ)/\fK \cong \Sp(2g,\bZ/4)/Y$ 
is isomorphic to the quotient $\bar H_0$ of
the group $H_0$ of Griess (described in Section \ref{se:extraspecial})
by its central subgroup of order two.
\end{prop}
\begin{proof}
Examine the extension
\begin{equation}\label{eq:extension}
1 \to \Gasp(2g,2)/\Gasp(2g,2,4) \to \Sp(2g,\bZ/4)/\Gasp(2g,2,4) \to \Sp(2g,2) \to 1.
\end{equation}
This is nonsplit, since the element of order two in $\Sp(2g,2)$ which swaps the
first basis vectors of $L$ and $L^*$ and fixes the remaining basis vectors
does not lift to an element of order two in $\Sp(2g,\bZ/4)/\Gasp(2g,2,4)$.

Let $E$ be the almost extraspecial group $O_2(H_0)$ of shape $2^{1+(2g+1)}$.
The action of $\Sp(2g,2)$ on $\Gasp(2g,2)/\Gasp(2g,2,4)\cong U$
is the same as the action of $\Out(E)'$ 
on $\Inn(E)$ (see \eqref{eq:Aut-almost-exspec}), namely the natural symplectic module.
It follows from the main theorem of Dempwolff \cite{Dempwolff:1974a} that
\[ H^2(\Sp(2g,2),\Gasp(2g,2)/\Gasp(2g,2,4)) \]
is one-dimensional. Thus $\Sp(2g,\bZ/4)/\Gasp(2g,2,4)$ is isomorphic to the group $\Aut(E)'$.

Since $\Gasp(2g,4)\subseteq \fK\subseteq \Gasp(2g,2,4)$ it follows that the short exact sequence
\[ 1 \to Z \to \fH \to \Sp(2g,2) \to 1 \]
also does not split. 
We have $\Gasp(2g,2,4)/\fK\cong \bZ/2$, and since $g\ge 4$, 
by Lemma \ref{le:H2Sp2gF2} we have $H^2(\Sp(2g,2),\bZ/2)=0$.
So $H^2(\Sp(2g,2),\Gasp(2g,2,4)/Y)=0$,
and hence $H^2(\Sp(2g,2),Z)$ is at most one-dimensional. 
Since we have a nonsplit extension \eqref{eq:extension},
it is exactly one-dimensional. The modules $E/[E,E]$ and $Z$ for $\Sp(2g,2)$ are both isomorphic
to the natural orthogonal module of dimension $2g+1$, so it follows that $\fH$
is isomorphic to $\bar H_0$.
\end{proof}

\begin{rk}
In the case $g=3$, Proposition \ref{pr:barH0} is still true, but needs a bit more work.
The group $H^2(\Sp(6,2),\bZ/2)$ is one-dimensional by Lemma \ref{le:H2Sp2gF2}, and we are left
with the nasty possibility that $\fH=\Sp(6,\bZ/4)/Y$ is isomorphic to a quotient of
the pullback of $\bar H_0 \to \Sp(6,2)$ and $\widetilde{\Sp(6,2)} \to \Sp(6,2)$
by the diagonal central element of order two. In order to prove that $\fH$ is
really isomorphic to $\bar H_0$ and not this other group, it suffices to construct a matrix representation of
a double cover of $\fH$ of dimension eight. Explicit matrices for this representation
were given in Section \ref{se:extraspecial}. On the other hand, the smallest faithful
irreducible complex representation in the case of the other possibility has dimension $64$.
It is worth noticing, though, that it does not matter which possibility is true, if we just
wish to prove the next theorem.
\end{rk}

\begin{lemma}\label{le:H2}
$H_2(\Sp(2g,\bZ))\cong\bZ$ for $g\ge 4$ and $H_2(\Sp(2g,\bZ))\cong\bZ \oplus \bZ/2$ for
$g=3$.
\end{lemma}
\begin{proof}
See for example Stein \cite[Theorem 2.2]{Stein:1975a}
for $g=3$, and 
Behr \cite[ Korollar~3.2]{Behr:1975a} together with 
Stein \cite[Theorem 5.3 and Remark 5 following
Corollary 5.5]{Stein:1971a} for $g\ge 4$.
See also Putman \cite[Theorem 5.1]{Putman:2012a} for a different
proof in the case $g\ge 4$.
\end{proof}

\begin{theorem}\label{th:H2Sp2gZ4}
For $g\ge 3$, we have $H_1(\fH)=0$. 
For $g\ge 4$, we have $H_2(\fH)\cong\bZ/2$, and for $g=3$, we have
$H_2(\fH)\cong\bZ/2\oplus \bZ/2$.
The map
\[ H_2(\Sp(2g,\bZ/{2^n})) \to H_2(\fH) \]
is an isomorphism for $n\ge 2$.
\end{theorem}
\begin{proof}
The computation of the abelianisation $H_1(\fH)$ is straightforward.
It follows from Propositions \ref{pr:surj1} and \ref{pr:surj2} that
for $n\ge 2$ the maps
\begin{equation}\label{eq:surjH2}
H_2(\Sp(2g,\bZ)) \to H_2(\Sp(2g,\bZ/2^n)) \to H_2(\Sp(2g,\bZ/4))\to H_2(\fH) \to
H_2(\Sp(2g,2)) 
\end{equation}
are surjective, and from Deligne's Theorem (see the introduction)
that the kernel of the first map contains every element divisible by two. Consulting
Lemma \ref{le:H2}, we see that $H_2(\Sp(2g,\bZ/2^n))$ and  $H_2(\fH)$
are quotients of the groups given.

By Proposition \ref{pr:barH0}, there is a nontrivial
element of $H_2(\fH,\bZ/2)$ which is killed by the map to $H_2(\Sp(2g,2))$.
Namely, the central extension $\tilde\fH\to \fH$
is not inflated from $\Sp(2g,2)$ because 
the kernel of $\tilde \fH \to \Sp(2g,2)$ is 
the nonabelian group $E$.

Comparing the value of $H_2(\Sp(2g,\bZ))$ given in Lemma \ref{le:H2}
with the value of $H_2(\Sp(2g,2))$ given in Lemma
\ref{le:H2Sp2gF2}, the theorem follows.
\end{proof}

\begin{cor}\label{co:fHtoSp}
We have $H^2(\fH,\bZ/2)\cong\bZ/2$. For $g\ge 3$ and $n\ge 2$, the map
\[ H^2(\fH, A) \to H^2(\Sp(2g,\bZ/{2^n}), A) \]
is an isomorphism for any abelian group of coefficients $A$ with trivial action.
\end{cor}
\begin{proof}
This follows from Theorem \ref{th:H2Sp2gZ4} and the universal coefficient theorem.
\end{proof}

\noindent
\phantomsection
\begin{center}
\textbf{\large Summary of homology and cohomology groups}\bigskip
\end{center}
\addcontentsline{toc}{section}{Summary of homology and cohomology groups}
\setcounter{equation}{0}

Values for $g\ge 4\begin{cases} \text{even} \\ \text{odd} \end{cases}$
are as follows:

\[ \renewcommand{\arraystretch}{1.4}
\begin{array}{c|ccccc}
\hline
\text{\rm Group}&H_1(-) & H_2(-)&H^2(-,\bZ)&H^2(-,\bZ/8)&H^2(-,\bZ/2) \\ \hline
\mcg{g} & 0&\bZ&\bZ&\bZ/8&\bZ/2 \\
\Sp(2g,\bZ)&0&\bZ&\bZ&\bZ/8&\bZ/2 \\
P\Sp(2g,\bZ)&0&\begin{cases}\bZ\oplus \bZ/2 \\ \bZ \end{cases} & \bZ &
\begin{cases} \bZ/8 \oplus \bZ/2 \\ \bZ/8 \end{cases} & 
\begin{cases} \bZ/2 \oplus \bZ/2 \\ \bZ/2 \end{cases}  \\
\Sp(2g,\bZ/4)&0 & \bZ/2 &0& \bZ/2 & \bZ/2 \\
P\Sp(2g,\bZ/4) & 0 & 
\begin{cases} \bZ/2 \oplus \bZ/2 \\ \bZ/4 \end{cases} &
0 & \begin{cases} \bZ/2 \oplus \bZ/2 \\ \bZ/4 \end{cases}&
 \begin{cases} \bZ/2 \oplus \bZ/2 \\ \bZ/2 \end{cases} \\
\fH = \Sp(2g,\bZ/4)/Y & 0 & \bZ/2 & 0 & \bZ/2 & \bZ/2  \\
\Sp(2g,2) & 0 & 0&0 &0&0 \\ \hline
\end{array} \]
\bigskip

Values for $g=3$:

\[ \renewcommand{\arraystretch}{1.4}
\begin{array}{c|ccccc}
\hline
\text{\rm Group}&H_1(-) & H_2(-)&H^2(-,\bZ)&H^2(-,\bZ/8)&H^2(-,\bZ/2) \\ \hline
\mcg{3} & 0&\bZ \oplus \bZ/2 &\bZ&\bZ/8\oplus \bZ/2&\bZ/2 \oplus \bZ/2 \\
\Sp(6,\bZ)&0&\bZ\oplus \bZ/2&\bZ&\bZ/8\oplus \bZ/2&\bZ/2\oplus\bZ/2 \\
P\Sp(6,\bZ)&0&\bZ\oplus \bZ/2 & \bZ &
\bZ/8 \oplus \bZ/2 & \bZ/2  \\
\Sp(6,\bZ/4)&0 & \bZ/2\oplus \bZ/2 &0& \bZ/2\oplus\bZ/2 & \bZ/2\oplus\bZ/2 \\
P\Sp(6,\bZ/4)&0& \bZ/4\oplus\bZ/2 & 0 & \bZ/4\oplus\bZ/2 & \bZ/2\oplus\bZ/2 \\
\fH = \Sp(6,\bZ/4)/Y & 0 & \bZ/2\oplus\bZ/2 & 0 & \bZ/2\oplus\bZ/2 & \bZ/2\oplus\bZ/2  \\
\Sp(6,2) & 0 & \bZ/2&0 &\bZ/2&\bZ/2 \\ \hline
\end{array} \]
\bigskip

\newpage

Values for $g=2$:

\[ \renewcommand{\arraystretch}{1.4}
\begin{array}{c|ccccc}
\hline
\text{\rm Group}&H_1(-) & H_2(-)&H^2(-,\bZ)&H^2(-,\bZ/8)&H^2(-,\bZ/2) \\ \hline
\mcg{2} & \bZ/10&\bZ/2 &\bZ/10&(\bZ/2)^2&(\bZ/2)^2 \\
\Sp(4,\bZ)&\bZ/2&\bZ\oplus \bZ/2&\bZ\oplus\bZ/2&\bZ/8\oplus(\bZ/2)^2&(\bZ/2)^3 \\
P\Sp(4,\bZ)&\bZ/2&\bZ\oplus(\bZ/2)^2 & \bZ\oplus\bZ/2 &
\bZ/8 \oplus(\bZ/2)^3 & (\bZ/2)^4  \\
\Sp(4,\bZ/4)&\bZ/2 & (\bZ/2)^2 &\bZ/2& (\bZ/2)^3 & (\bZ/2)^3 \\
P\Sp(4,\bZ/4)&\bZ/2& (\bZ/2)^3 & \bZ/2 & (\bZ/2)^4 & (\bZ/2)^4 \\
\fH = \Sp(4,\bZ/4)/Y & \bZ/2 & (\bZ/2)^2 & \bZ/2 & (\bZ/2)^3 & (\bZ/2)^3  \\
\Sp(4,2) & \bZ/2 & \bZ/2&\bZ/2 &(\bZ/2)^2&(\bZ/2)^2 \\ \hline
\end{array} \]
\bigskip

Values for $g=1$:

\[ \renewcommand{\arraystretch}{1.4}
\begin{array}{c|ccccc}
\hline
\text{\rm Group}&H_1(-) & H_2(-)&H^2(-,\bZ)&H^2(-,\bZ/8)&H^2(-,\bZ/2) \\ \hline
\mcg{1} & \bZ/12&0 &\bZ/12&\bZ/4&\bZ/2 \\
\Sp(2,\bZ)&\bZ/12&0&\bZ/12&\bZ/4&\bZ/2 \\
P\Sp(2,\bZ)&\bZ/6&0& \bZ/6&\bZ/2& \bZ/2 \\
\Sp(2,\bZ/4)&\bZ/4&\bZ/2&\bZ/4& \bZ/4\oplus\bZ/2 &  (\bZ/2)^2 \\
P\Sp(2,\bZ/4)&\bZ/2&\bZ/2&\bZ/2& (\bZ/2)^2 & (\bZ/2)^2 \\
\fH = \Sp(2,\bZ/4)/Y &\bZ/4& \bZ/2 & \bZ/4 &\bZ/4\oplus\bZ/2 &  (\bZ/2)^2  \\
\Sp(2,2) & \bZ/2 & 0&\bZ/2 &\bZ/2&\bZ/2 \\ \hline
\end{array} \]
\bigskip

\newpage

%\bibliographystyle{amsplain}
%\bibliography{../repcoh}

\begin{thebibliography}{10}

\bibitem{Behr:1975a}
H.~Behr, \emph{{Explizite Pr\"asentation von Chevalleygruppen \"uber $\mathbb
  Z$}}, Math.\ Zeit. \textbf{141} (1975), 235--241.

\bibitem{Benson:theta}
D.~J. Benson, \emph{{Theta functions and a presentation of
  $2^{1+(2g+1)}\mathsf{Sp}(2g,2)$}}, preprint, 2017.\newline
\href{http://homepages.abdn.ac.uk/d.j.benson/html/archive/benson.html}{http://homepages.abdn.ac.uk/d.j.benson/html/archive/benson.html}

\bibitem{Bouc/Mazza:2004a}
S.~Bouc and N.~Mazza, \emph{{The Dade group of (almost) extraspecial
  $p$-groups}}, J.~Algebra \textbf{192} (2004), 21--51.

\bibitem{Carlson/Thevenaz:2000a}
J.~F. Carlson and J.~Th\'evenaz, \emph{{Torsion endo-trivial modules}},
  Algebras and Representation Theory \textbf{3} (2000), 303--335.

\bibitem{Atlas}
J.~H. Conway, R.~T. Curtis, S.~P. Norton, R.~A. Parker, and R.~A. Wilson,
  \emph{{Atlas of Finite Groups}}, Oxford University Press, 1985.

\bibitem{Deligne:1978a}
P.~Deligne, \emph{{Extensions centrales non r\'esiduellement finies des groupes
  arithm\'etiques}}, Comptes Rendus Acad.\ Sci.\ Paris, S\'erie I \textbf{287}
  (1978), no.~4, A203--A208.

\bibitem{Dempwolff:1974a}
U.~Dempwolff, \emph{{Extensions of elementary abelian groups of order $2^{2n}$
  by $S_{2n}(2)$ and the degree $2$-cohomology of $S_{2n}(2)$}}, Illinois J.\
  Math. \textbf{18} (1974), 451--468.

\bibitem{Diaconis:2010a}
P.~Diaconis, \emph{{Threads through group theory}}, Character Theory of Finite
  Groups (Mark~L. Lewis, Gabriel Navarro, Donald~S. Passman, and Thomas~R.
  Wolf, eds.), Contemp.\ Math., vol. 524, American Math.\ Society, 2010,
  pp.~33--47.

\bibitem{Earle/Eells:1969a}
C.~J. Earle and J.~Eells, \emph{{A fibre bundle description of Teichm\"uller
  theory}}, J.\ Diff.\ Geom. \textbf{3} (1969), 19--43.

\bibitem{Endo:1998a}
H.~Endo, \emph{{A construction of surface bundles over surfaces with non-zero
  signature}}, Osaka J.\ Math. \textbf{35} (1998), 915--930.

\bibitem{Funar/Pitsch:preprint}
L.~Funar and W.~Pitsch, \emph{{Finite quotients of symplectic groups vs mapping
  class groups}}, Preprint, 2016.

\bibitem{Glasby:1995a}
S.~P. Glasby, \emph{{On the faithful representations, of degree $2^n$, of
  certain extensions of $2$-groups by orthogonal and symplectic groups}},
  J.~Austral.\ Math.\ Soc. \textbf{58} (1995), 232--247.

\bibitem{Gocho:1990a}
T.~Gocho, \emph{{The topological invariant of three-manifolds based on the
  $U(1)$ gauge theory}}, Proc.\ Japan Acad., Ser.\ A \textbf{66} (1990),
  237--239.

\bibitem{Gocho:1990b}
\bysame, \emph{{The topological invariant of three-manifolds based on the
  $U(1)$ gauge theory}}, J.\ Fac.\ Sci.\ Univ.\ Tokyo Sect IA, Math.
  \textbf{39} (1992), 169--184.

\bibitem{Gorenstein:1968a}
D.~Gorenstein, \emph{{Finite groups}}, Harper \& Row, 1968.

\bibitem{Griess:1973a}
R.~L. Griess, \emph{{Automorphisms of extra special groups and nonvanishing
  degree $2$ cohomology}}, Pacific J.\ Math. \textbf{48} (1973), no.~2,
  403--422.

\bibitem{Hall/Higman:1956a}
P.~Hall and G.~Higman, \emph{{On the $p$-length of $p$-soluble groups and
  reduction theorems for Burnside's problem}}, Proc.\ London Math.\ Soc.
  \textbf{6} (1956), 1--42.

\bibitem{Hambleton/Korzeniewski/Ranicki:2007a}
I.~Hambleton, A.~Korzeniewski, and A.~Ranicki, \emph{{The signature of a fibre
  bundle is multiplicative mod $4$}}, Geometry \& Topology \textbf{11} (2007),
  251--314.

\bibitem{Hamstrom:1966a}
M.-E. Hamstrom, \emph{{Homotopy groups of the space of homeomorphisms on a
  $2$-manifold}}, Illinois J.\ Math. \textbf{10} (1966), 563--573.

\bibitem{Hiss:1984a}
G.~Hiss, \emph{{Die adjungierten Darstellungen der Chevalley-Gruppen}}, Arch.\
  Math.\ (Basel) \textbf{42} (1984), 408--416.

\bibitem{Huppert:1967a}
B.~Huppert, \emph{{Endliche Gruppen I}}, Grundlehren der mathematischen
  Wissenschaften, vol. 134, Springer-Verlag, Ber\-lin/New York, 1967.

\bibitem{Igusa:1964a}
J.-I. Igusa, \emph{{On the graded ring of theta constants}}, Amer.\ J.\ Math.
  \textbf{86} (1964), 219--246.

\bibitem{Korzeniewski:2005a}
A.~Korzeniewski, \emph{{On the signature of fibre bundles and absolute
  Whitehead torsion}}, {Ph.\ D.\ Dissertation}, University of Edinburgh, 2005.

\bibitem{Lam/Smith:1989a}
T.~Y. Lam and T.~Smith, \emph{{On the Clifford-Littlewood-Eckmann groups: A new
  look at periodicity mod $8$}}, Rocky Mountain J.\ Math. \textbf{19} (1989),
  no.~3, 749--786.

\bibitem{Luke/Mason:1972a}
R.~Luke and W.~K. Mason, \emph{{The space of homeomorphisms on a compact
  two-manifold is an absolute neighborhood retract}}, Trans.\ Amer.\ Math.\
  Soc. \textbf{164} (1972), 275--285.

\bibitem{Meyer:1973a}
W.~Meyer, \emph{{Die Signatur von Fl\"achenb\"undeln}}, Math.\ Ann.
  \textbf{201} (1973), 239--264.

\bibitem{Nebe/Rains/Sloane:2001a}
G.~Nebe, E.~M. Rains, and N.~J.~A. Sloane, \emph{{The invariants of the
  Clifford groups}}, Designs, Codes and Cryptography \textbf{24} (2001),
  99--121.

\bibitem{Newman/Smart:1964a}
M.~Newman and J.~R. Smart, \emph{{Symplectic modulary groups}}, Acta
  Arithmetica \textbf{IX} (1964), 83--89.

\bibitem{Putman:2012a}
A.~Putman, \emph{{The Picard group of the moduli space of curves with level
  structures}}, Duke Math.\ J. \textbf{161} (2012), 623--674.

\bibitem{Quillen:1971d}
D.~G. Quillen, \emph{{The mod $2$ cohomology rings of extra-special $2$-groups
  and the spinor groups}}, Math.\ Ann. \textbf{194} (1971), 197--212.

\bibitem{Rovi:AGT}
C.~Rovi, \emph{{The nonmultiplicativity of the signature modulo $8$ of a fibre
  bundle is an Arf--Kervaire invariant}}, Algebr. Geom. Topol. \textbf{18} (2018), 1281--1322. 

\bibitem{Runge:1993a}
B.~Runge, \emph{{On Siegel modular forms, I}}, J.~reine \& angew.\ Math.
  \textbf{436} (1993), 57--85.

\bibitem{Runge:1995a}
\bysame, \emph{{On Siegel modular forms, II}}, Nagoya Math.\ J. \textbf{138}
  (1995), 179--197.

\bibitem{Runge:1996a}
\bysame, \emph{{Codes and Siegel modular forms}}, Discrete Mathematics
  \textbf{148} (1996), 175--204.

\bibitem{Sato:2010a}
M.~Sato, \emph{{The abelianization of the level $d$ mapping class group}},
  J.~Topology \textbf{3} (2010), 847--882.

\bibitem{Schmid:2000a}
P.~Schmid, \emph{{On the automorphism group of extraspecial $2$-groups}},
  J.~Algebra \textbf{234} (2000), 492--506.

\bibitem{Stancu:2002a}
R.~Stancu, \emph{{Almost all generalized extraspecial $p$-groups are
  resistant}}, J.~Algebra \textbf{249} (2002), 120--126.

\bibitem{Stein:1971a}
M.~R. Stein, \emph{{Generators, relations and coverings of Chevalley groups
  over commutative rings}}, Amer.\ J.\ Math. \textbf{93} (1971), 965--1004.

\bibitem{Stein:1975a}
M.~R. Stein, \emph{{The Schur multipliers of $\mathrm{Sp}_6(\mathbb Z)$,
  $\mathrm{Spin}_8(\mathbb Z)$, $\mathrm{Spin}_7(\mathbb Z)$ and
  $\mathrm{F}_4(\mathbb Z)$}}, Math.\ Ann. \textbf{215} (1975), 165--172.

\bibitem{Steinberg:1981a}
R.~Steinberg, \emph{{Generators, relations and coverings of algebraic groups
  II}}, J.~Algebra \textbf{71} (1981), 527--543.

\bibitem{Tsushima:2003a}
K.~Tsushima, \emph{{On a decomposition of Bruhat type for a certain finite
  group}}, Tsukuba J.\ Math. \textbf{27} (2003), no.~2, 307--317.

\end{thebibliography}
\iftrue
\newcommand{\noopsort}[1]{}
\providecommand{\bysame}{\leavevmode\hbox to3em{\hrulefill}\thinspace}
\providecommand{\MR}{\relax\ifhmode\unskip\space\fi MR }
\providecommand{\MRhref}[2]{%
  \href{http://www.ams.org/mathscinet-getitem?mr=#1}{#2}
}
\providecommand{\href}[2]{#2}

\fi

\end{document}